\documentclass[a4paper,12pt,onecolumn]{article}
\usepackage{etex}

\usepackage{hyperref}
\usepackage[vmargin=2cm,hmargin=2cm,headheight=14.5pt,top=2cm,headsep=.5cm]{geometry}

\usepackage{bm}
\usepackage{empheq}
\usepackage{stackrel}
\usepackage{cases}
\usepackage{mathtools}
\usepackage{amsthm,amsmath,amscd}
\usepackage{makeidx}
\usepackage[all]{xy}
\usepackage{framed}
\usepackage[ddmmyy]{datetime}
\usepackage{tikz-cd}
\tikzcdset{every label/.append style = {font = \small}}
\usepackage{marvosym}
\newcommand*{\footnotemarkcolor}{red}
\makeatletter
\renewcommand*{\@makefnmark}{\hbox{\@textsuperscript{%
   \color{\footnotemarkcolor}\normalfont\@thefnmark}}}
\makeatother
\makeatletter
\def\@fnsymbol#1{\ensuremath{\ifcase#1\or \text{\Mercury}
\or \text{\Venus} \or \text{\Earth} \or
 \text{\Jupiter} \or \text{\Saturn} \or \text{\Neptune} \or \text{\Uranus} \or \text{\Pluto}
   \or \text{\Moon} \or \text{\Sun}
\else\@ctrerr\fi}}%
\makeatother
\usepackage{graphicx}
\DeclareMathSizes{12}{12}{8}{6}

\usepackage{cite}
\usepackage{url}
\usepackage[charter]{mathdesign}
\usepackage{accents}

\newtheoremstyle{ptheorem}{1em}{0em}{\itshape}{}{\bfseries}{.}{.5em}{\thmname{#1}\thmnumber{ #2}\thmnote{ (\hspace{-.01pt}{#3})}}

\theoremstyle{ptheorem}

\newtheorem{thm}{Theorem}[section]

\newtheorem{lem}[thm]{Lemma}

\newtheoremstyle{hdef}{1em}{0em}{}{}{\bfseries}{.}{.5em}{\thmname{#1}\thmnumber{ #2}\thmnote{ (\hspace{-.01pt}{#3})}}
\theoremstyle{hdef}

\newtheorem{rem}[thm]{Remark}

\makeatletter
\newtheoremstyle{premark}{1em}{0em}{
\addtolength{\@totalleftmargin}{1.5em}
\addtolength{\linewidth}{-1.5em}
\parshape 1 1.5em \linewidth}{}{\scshape}{.}{.5em}{}
\makeatother

\theoremstyle{premark}

\numberwithin{equation}{section}
\numberwithin{figure}{section}

\DeclareMathOperator{\dif}{d}

\newcommand{\cC}{{\mathcal C}}

\newcommand{\bN}{{\mathbb N}}

\newcommand{\bR}{{\mathbb R}}

\newcommand{\bZ}{{\mathbb Z}}

\renewcommand{\c}{\gamma}

\newcommand{\e}{\epsilon}

\newcommand{\ol}{\overline}

\renewcommand{\d}{\delta}

\renewcommand{\(}{\left(}
\renewcommand{\)}{\right)}

\newcommand{\lil}{\lim\limits}

\newcommand{\til}{\widetilde}
\newcommand{\Lsp}[1]{\operatorname{L^{#1}}}
\newcommand{\olb}[1]{%
  \vbox{\offinterlineskip\ialign{\hfil##\hfil\cr $\rotatebox[origin=c]{90}{$]$}$\cr\noalign{\kern-.45ex}{$#1$}\cr}}}
\allowdisplaybreaks
\begin{document}
\title{Existence of solutions of integral equations with asymptotic conditions\footnote{Partially supported by Xunta de Galicia (Spain), project EM2014/032 and AIE Spain and FEDER, grants MTM2013-43014-P, MTM2016-75140-P.}}

\author{
Alberto Cabada\\
\normalsize e-mail: alberto.cabada@usc.es\\
Luc\'ia L\'opez-Somoza\footnote{Supported by  FPU scholarship, Ministerio de Educaci\'on, Cultura y Deporte, Spain.} \\
\normalsize e-mail: lucia.lopez.somoza@usc.es\\
F. Adri\'an F. Tojo \\
\normalsize e-mail: fernandoadrian.fernandez@usc.es\\
\normalsize \emph{Instituto de Ma\-te\-m\'a\-ti\-cas, Facultade de Matem\'aticas,} \\ \normalsize\emph{Universidade de Santiago de Com\-pos\-te\-la, Spain.}\\ 
}
\date{}

\maketitle

%\begin{flushright}{\color{red}\textbf{\textit{Versi\'on: \jobname\quad Fecha: \today\ --\  \currenttime}}}\end{flushright}

\begin{abstract}
In  this work we will consider integral equations defined on the whole real line and look for solutions which satisfy some certain kind of asymptotic behavior. To do that, we will define a suitable Banach space which, to the best of our knowledge, has never been used before. In order to obtain fixed points of the integral operator, we will consider the fixed point index theory and apply it to this new Banach space.
\end{abstract}

\renewcommand{\abstractname}{Acknowledgements}
\begin{abstract}
	The authors want to acknowledge his gratitude towards Prof. Santiago Codesido for his insights concerning the projectile equation. 
\end{abstract}

%\noindent {\bf Keywords: Green's functions, PDEs, linear involution, heat equation.}  

\section{Introduction}
In this paper we study the existence of fixed points of integral operators of the form
\[Tu(t)=p(t)+\int_{-\infty}^{\infty}k(t,s)\,\eta(s)\,f(s,u(s))\, \dif s.\]

There are many results in the recent literature in which the authors deal with differential or integral problems in unbounded intervals (see for instance \cite{MinCar,MinCar2,MinCar-4ord,FiMinCar,DjeGue} and the references therein). The main difficulties which appear while dealing with this kind of problems arise as a consequence of the lack of compactness of the operator. In all of the cited references the authors solve this problem by means of the following relatively compactness criterion (see \cite{Corduneanu,przerad}) which involves some stability condition at $\pm\infty$:

\begin{thm}[{\cite[Theorem 1]{przerad}}] \label{thm_equic}
Let $E$ be a Banach space and $\cC(\mathbb{R},E)$ the space of all bounded continuous functions $x\colon \bR\rightarrow E$. For a set $D\subset \cC(\bR,E)$ to be relatively compact, it is necessary and sufficient that:
\begin{enumerate}
	\item $\{x(t), \ x\in D\}$ is relatively compact in $E$ for any $t\in\bR$;
	\item for each $a>0$, the family $D_a:=\{x|_{[-a,a]}, \ x\in D \}$ is equicontinuous;
	\item $D$ is stable at $\pm\infty$, that is, for any $\varepsilon>0$, there exists $T>0$ and $\delta>0$ such that if $\|x(T)-y(T)\|\le \delta$, then $\|x(t)-y(t)\|\le \varepsilon$ for $t\ge T$ and if $\|x(-T)-y(-T)\|\le \delta$, then $\|x(t)-y(t)\|\le \varepsilon$ for $t\le -T$, where $x$ and $y$ are arbitrary functions in $D$.
\end{enumerate}
\end{thm}

By using the previous result, the authors of the aforementioned references prove the existence of solutions of differential or integral problems by means of either Schauder's fixed point Theorem or lower and upper functions method. 

In this paper, we will deal with the problem of compactness of the integral operator using a different strategy: we will define a suitable Banach space, which will be proved to be isometric isomorphic to the space 
\[\cC^n(\ol\bR,\bR):=\left\{f:\ol\bR\to\bR\ :\ f|_{\bR}\in\cC^n(\bR,\bR),\ \exists \lim_{t\to\pm\infty}f^{(j)}(t) \in \bR,\ j=0,\dots,n\right\}.\]
This isomorphism will allow us to apply Arcel\`a-Ascoli's Theorem to our Banach space instead of using Theorem \ref{thm_equic}. 

Moreover, the Banach space that we will define will include some asymptotic condition which will ensure a certain asymptotic behavior of the solutions of the problem. Later on, we will use index theory in general cones \cite{FigToj} to obtain the desired fixed points.

The paper is divided in the following way: in Section 2 we present a physical problem which motivates the importance of the asymptotic behavior of solutions of a differential equation. In Section 3 we first summarize classical definitions of asymptotic behavior and then define a suitable Banach space and study its properties. Section 4 includes results of existence of fixed points of integral equations by means of the theory of fixed point index in cones. Finally, in Section 5 we will reconsider the physical problem presented in Section 2 and we will solve it by using the results given in Section 4.

\section{Motivation}

In many contexts it is interesting to anticipate the asymptotic behavior of the solution of a differential problem. For instance,  consider the classical projectile equation that describes the motion of an object that is launched vertically from the surface of a planet towards deep space \cite{Holmes},
\begin{equation}\label{eqorr}
u''(t)=-\frac{g\,R^2}{(u(t)+R)^2},\ t\in[0,\infty);\ u(0)=0,\ u'(0)=v_0,
\end{equation}
where $u$ is the distance from the surface of the planet, $R$ is the radius of the planet, $g$ is the surface gravity constant and $v_0$ the initial velocity. Clearly, if $v_0$ is not big enough, the projectile will reach a maximum height, at which $u'$ will be zero, and then fall. Hence, in order to compute the minimum velocity necessary for the projectile to escape the planet's gravity, it is enough to consider that $u(t)\to\infty$ and $u'(t)\to 0$. Then, multiplying both sides of \eqref{eqorr} by $u'$ and integrating between $0$ and $t$,
\[\frac{1}{2}[(u'(t))^2-v_0^2]=g\,R^2\left[\frac{1}{R+u(t)}-\frac{1}{R}\right].\]
Thus, taking the limit when $t\to\infty$, $-v_0^2/2=-gR$, that is, the scape velocity is $v_s=\sqrt{2gR}$. Observe that, with $v_0=v_s$, we have
\[u'(t)=\sqrt\frac{2gR^2}{u(t)+R}.\]
Using the same argument, for any initial velocity higher than $v_s$, when the projectile is far enough from the planet, it should drift away at constant velocity given by $v_\infty=\sqrt{v_0^2-2gR}$.

Now, the solution of \eqref{eqorr} has a interesting asymptotic behavior. For $v_0>v_s$, it is asymptotically linear as was previously said. This can be checked using L'Hopital's rule.
 \[\lim_{t\to\infty}\frac{u(t)}{t}=\lim_{t\to\infty}u'(t)=v_\infty.\]
 In the particular case $v_0=v_s$ we have that $v_\infty=0$ and
\begin{align*}\lim_{t\to\infty}\frac{u(t)}{t^\frac{2}{3}} & =\left[\lim_{t\to\infty}\frac{u(t)^\frac{3}{2}}{t}\right]^\frac{2}{3}=\left[\frac{3}{2}\lim_{t\to\infty}u(t)^\frac{1}{2}u'(t)\right]^\frac{2}{3}=\left[\frac{3}{2}\lim_{t\to\infty}u(t)^\frac{1}{2}\sqrt\frac{2gR^2}{u(t)+R}\right]^\frac{2}{3} \\ & =\left[\frac{3}{2}\sqrt{2gR^2}\right]^\frac{2}{3}=\(\frac{3}{2}\)^\frac{2}{3}\sqrt[3]{2gR^2}.\end{align*}

In a more realistic setting, with a self propelled projectile, we could consider 
\begin{equation}\label{eqorr2}
u''(t)=-\frac{gR^2}{(u(t)+R)^2}+h(t,u(t))-\rho(u(t))u'(t),\ t\in[0,\infty);\ \ u(0)=0,\ u'(0)=v_0,
\end{equation}
where $h(t,y)$ is the acceleration generated by the propulsion system of the rocket (which depends on time and also height, since different phases of the launch require different propulsion systems) and $\rho$ is the friction coefficient, which depends on height since it relates to atmospheric drag. The friction term is expected to not affect the asymptotic behavior of the solution (the atmosphere is finite, and therefore $\rho$ has compact support), so it would be interesting to study for what kinds of $h$ when it would be reasonable to expect the same asymptotic behavior as that of the solution of \eqref{eqorr}. In any case, we would have to define first what we  understand by \emph{asymptotic behavior}.

\section{Asymptotic behavior}
\subsection{Classical ways of dealing with asymptotic behavior}
Asymptotic behavior, always associated to perturbation theory in Physics, has been studied for a long time in an abstract mathematical  way. For instance, if we go to the book of G. H. Hardy \emph{Orders of Infinity} \cite{Hardy}, we find the following notions:
\begin{quote}
\begin{itshape}
``Let us suppose that $f$ and $\varphi$ are two functions of the continuous 
variable $x$, defined for all values of $x$ greater than a given value $x_0$. 
Let us suppose further that $f$ and $\varphi$ are positive, continuous, and 
steadily increasing functions which tend to infinity with $x$; and let us 
consider the ratio $f/\varphi$. We must distinguish four cases:
\begin{itemize}
\item If $f/\varphi\to\infty$ with $x$, we shall say that the \emph{rate of increase}, or 
simply the \emph{increase}, of $f$ is greater than that of $\varphi$, and shall write
\[f\succ\varphi.\]
\item If $f/\varphi\to 0$, we shall say that the increase of $f$ is less than that 
of $\varphi$, and write 
\[f\prec\varphi.\]
\item If $f/\varphi$ remains, for all values of $x$ however large, between two 
fixed positive numbers $\delta$, $\Delta$, so that $0 < \delta <f/\varphi < \Delta$, we shall say that 
the increase of $f$ is equal to that of $\varphi$ , and write 
\[f\asymp\varphi.\]
It may happen, in this case, that $f/\varphi$ actually tends to a definite 
limit. If this is so, we shall write 
\[f\mathrlap{\ -}\asymp\varphi.\]
Finally, if this limit is \emph{unity}, we shall write 
\[f\sim\varphi.\]
\item If a positive constant $\d$ can be found such that $f>\d\varphi$ for all 
sufficiently large values of $x$, we shall write 
\[f\succcurlyeq\varphi;\]
and if a positive constant $\Delta$ can be found such that $f<\Delta\varphi$ for all 
sufficiently large values of $x$, we shall write 
\[f\preccurlyeq\varphi.\text{''}\]
\end{itemize}
\end{itshape}
\end{quote}

Hence, it is clear that there are several ways to approach this issue. The case of $f\preccurlyeq\varphi$ (also written  as $f=O(\varphi)$ in the notation of Landau) is the one used in the study of computational complexity \cite{Sipser}. \par
On the other hand, we find this kind of asymptotic behavior in \emph{fading memory spaces} \cite{Kuang}, but also in \emph{weighted spaces} \cite{SinMa}, where the comportment   can also be that associated to $f\prec\varphi$, noted as $f=o(\varphi)$ as well \cite{Sipser}.

The aforementioned notions of asymptotic behavior are connected trough the exponential map to their corresponding ones using the difference instead of the quotient. To be explicit, consider the exponential map
	\begin{center}\begin{tikzcd}[row sep=tiny]
	\cC(\bR,\bR) \arrow{r}{\operatorname{exp}} & \cC(\bR,\bR^+) \\
f\arrow[hook]{r} & e^f
	\end{tikzcd}
	\end{center}
 where $\bR^+=(0,\infty)$. Thus, for every $f,\varphi\in\cC(\bR,\bR)$,
\begin{itemize}
\item $\lim\limits_{x\to\infty}(f-\varphi)=\infty$ if and only if $e^f\succ e^\varphi$.

\item $\lim\limits_{x\to\infty}(f-\varphi)=-\infty$ if and only if $e^f\prec e^\varphi$.
\item  $|f-\varphi|$ is bounded if and only if $e^f\asymp e^\varphi$.
\item $\lim\limits_{x\to\infty}(f-\varphi)=L\in\bR$ if and only if $e^f\mathrlap{\ -}\asymp e^\varphi$.
\item $\lim\limits_{x\to\infty}(f-\varphi)=0$ if and only if $ e^f\sim\e^\varphi$.
\item A constant $\d\in\bR$ can be found such that $f-\varphi>\d$ for all 
sufficiently large values of $x$  if and only if $e^f\succcurlyeq e^\varphi$.
\item A constant $\Delta\in\bR$ can be found such that $f-g<\Delta$ for all 
sufficiently large values of $x$ if and only if $e^f\preccurlyeq e^\varphi$.
\end{itemize}

Needless to say, the all of the aforementioned definitions can be applied to non necessarily positive functions with due precautions.

In this work we will center our discussion in the case $f\mathrlap{\ -}\asymp\varphi$. In order to to so, we will need a conveniently defined Banach space which is not among the mentioned above.

\subsection{The space of continuously $\boldmath{n}$-differentiable $\boldmath{ \varphi}$-extensions to infinity}
Consider the space $\ol \bR:=[-\infty,\infty]$ with the compact topology, that is, the topology generated by the basis
\[\{B(a,r)\ :\ a\in\bR,\ r\in\bR^+\}\cup\{[-\infty,a)\ :\ a\in\bR\}\cup\{(a,\infty]\ :\ a\in\bR\}.\]
With this topology, $\ol \bR$ is homeomorphic to any compact interval of $\bR$ with the relative topology inherited from the usual topology of $\bR$.\par
It is easy to check that $\cC(\ol\bR,\bR)$ is a Banach space with the usual supremum norm. We define, in a similar way,
\[\cC^n(\ol\bR,\bR):=\left\{f:\ol\bR\to\bR\ :\ f|_{\bR}\in\cC^n(\bR,\bR),\ \exists \lim_{t\to\pm\infty}f^{(j)}(t) \in \bR,\ j=0,\dots,n\right\},\]
 for $n\in\bN$. $\cC^n(\ol\bR,\bR)$, $n\in\bN$, is a Banach space with the norm \[\|f\|_{(n)}:=\sup\left\{\left\|f^{(k)}\right\|_\infty\ :\ k=0,\dots,n\right\}.\]
Take now $\varphi\in\cC^n(\bR,\bR^+)$ and define the \emph{space of continuously $n$-differentiable $ \varphi$-extensions to infinity}
\[\widetilde\cC^n_\varphi\equiv\widetilde\cC^n_\varphi(\bR,\bR)=\left\{f\in\cC^n(\bR,\bR)\ :\ \exists \til f\in\cC^n(\ol\bR,\bR),\  f=\varphi \left(\til f|_{\bR}\right)\right\}.\]
We define the norm
\[\|f\|_\varphi:=\|\til f\|_{(n)},\ f\in\widetilde\cC_\varphi.\]
$\|\cdot\|_\varphi$ is well defined, since the extension $\til f$ is unique for every $f$; indeed, assume there are $\til f_1$, $\til f_2$ such that $\til f_1\,\varphi=\til f_2\,\varphi=f$ in $\bR$. Since $\bR$ is dense in $\ol \bR$ and $\til f_1$ and $\til f_2$ are continuous, $\til f_1=\til f_2$.\par
On the other hand, for every $\til f\in\cC^n(\ol\bR,\bR)$ there exists a unique $f\in\widetilde\cC_\varphi$ such that $\til f|_{\bR}\,\varphi=f$ (just define $f:=\til f\,\varphi$ in $\bR$). 

This shows that there is an isometric isomorphism
\begin{equation*}\begin{split}
\Phi:\cC^n(\ol\bR,\bR)&\to\widetilde\cC^n_\varphi \\
\tilde{f}&\mapsto \Phi(\tilde{f})=\tilde{f}|_{\bR}\,\varphi,
\end{split}\end{equation*}
whose inverse isomorphism is 
\begin{equation*}\begin{split}
\Phi^{-1}:\widetilde\cC^n_\varphi &\to \cC^n(\ol\bR,\bR)\\
f&\mapsto \Phi^{-1}(f)=f/\varphi .
\end{split}\end{equation*}

Furthermore, Arcel\`a-Ascoli's Theorem applies to $\cC^n(\ol\bR,\bR)$ since $\ol\bR$ is a Hausdorff compact topological space and $\bR$ is a complete metric space. Using $\Phi$ we can apply the Theorem to $\widetilde\cC^n_\varphi$. To be precise,

\begin{thm}[Arcel\`a-Ascoli \cite{kelley}] Let $X$ be a Hausdorff compact topological space and $Y$ a complete metric space, and consider $\cC(X,Y)$ with the topology of the uniform convergence. Then $F\subset\cC(X,Y)$ has compact closure if and only if 
\begin{itemize}
\item $F(x)$ has compact closure for each $x\in X$, and
\item $F$ is equicontinuous. 
\end{itemize}
\end{thm}

If we write this Theorem in terms of $\til \cC^n_\varphi$ using the isomorphism $\Phi$ we get the following Theorem.

\begin{thm}
$F\subset \til \cC_\varphi^n$ has compact closure if and only if the two following conditions are satisfied:
\begin{itemize}
	\item  For each $t\in\bR$ the set $\{\til f(t), \ f\in F\}$ has compact closure or, which is the same (since $\til f(t) \in \bR$), $\{\til f(t), \ f\in F\}$ is bounded, that is, for each $t\in\bR$ there exists some constant $M>0$ such that
	\[\left|\frac{\partial^j \til f}{\partial t^j}(t)\right|=\left|\frac{\partial^j (f/\varphi)}{\partial t^j}(t)\right|\le M<\infty,\]
	for all $j=0,\dots,n$ and $f\in F$.
	\item $F$ is equicontinuous, that is, for all $\varepsilon\in \bR^+$ there exists some $\delta\in\bR^+$ such that 
	\[\left|\frac{\partial^j \til f}{\partial t^j}(r)- \frac{\partial^j \til f}{\partial t^j}(s)\right|=\left|\frac{\partial^j (f/\varphi)}{\partial t^j}(r)- \frac{\partial^j (f/\varphi)}{\partial t^j}(s)\right|<\varepsilon,\]
	for all $j=0,\dots,n$, $f\in F$ and $r,\,s\in\bR$ such that $|r-s|<\delta$.
\end{itemize}
\end{thm}
\begin{proof} Let $\til F^{(j)}:=\{(f/\varphi)^{(j)}\ :\ f\in F\}\subset\cC(\ol\bR,\bR)$, $j=0,\dots,n$.

Since
\[\|f\|_{(n)}:=\sup\left\{\left\|f^{(k)}\right\|_\infty\ :\ k=0,\dots,n\right\},\]
 $F$ has compact closure in $\til\cC_\varphi^n$ if and only if $\til F^{(j)}$ have compact closure in $\cC(\ol\bR,\bR)$ for $j=0,\dots,n$. By Arcel\`a-Ascoli Theorem, this happens if and only if
\begin{itemize}
	\item  for each $t\in\bR$ the set $\{ f(t)\ :\ f\in F^{(j)}\}$ has compact closure for $j=0,\dots,n$;
	\item $\til F^{(j)}$ is equicontinuous for $j=0,\dots,n$.
\end{itemize}
\end{proof}

 \begin{rem}\label{remext}Observe that, if $f\in\cC(\ol \bR,\bR)$ and $f|_\bR\in\cC^n(\bR,\bR)$ then $\lil_{t\to\pm\infty}f^{(k)}(t)=0$ for every $k=1,\dots,n$ since $f$ is asymptotically constant. Hence, $f\in \cC^n(\ol \bR,\bR)$.
\end{rem}
\begin{rem} Although  $\cC^n(\ol\bR,\bR)$ and $\widetilde\cC^n_\varphi$ are isometric isomorphic as Banach spaces,  $\cC^n(\ol\bR,\bR)$ is a Banach algebra but $\widetilde\cC^n_\varphi$ is not. In fact, we have that $\widetilde\cC^n_\varphi$ is a $\cC^n(\bR,\bR)$-module satisfying
	\[\|fg\|_\varphi \le\max_{j=0,\dots,n}\, \sum_{k=0}^j\binom{j}{k}\, \|(f/\varphi)^{(k)}\|_{\infty}\,  \|g^{(j-k)}\|_{\infty} \le 2^n\ \|f/\varphi\|_{(n)}\, \|g\|_{(n)} =2^n\,\|f\|_{\varphi}\,\|g\|_{(n)}\]
	for every $f\in \widetilde\cC^n_\varphi,\ g\in \cC^n(\bR,\bR)$.
	\end{rem}
We can extend the above definitions to more particular settings. Let $a,b\in\bR$ and consider
\begin{align*}\cC^n_{a,b}(\ol\bR,\bR):= & \{f\in\cC^n(\ol\bR,\bR)\ :\ f(-\infty)=a,\ f(\infty)=b\},\\
\widetilde\cC^{n}_{\varphi,a,b}:= & \{f\in\cC^n_{a,b}(\bR,\bR)\ :\ \exists \til f\in\cC^n(\ol\bR,\bR),\  f=\varphi\til f|_{\bR}\}.\end{align*}
$\cC^n_{a,b}(\ol\bR,\bR)$ is a closed subspace of $\cC(\ol\bR,\bR)$ so $\widetilde\cC_\varphi^{a,b}$ is a Banach subspace  of $\til\cC_\varphi$.

Similarly, we can work on intervals of the form $[a,\infty)$ (or $(-\infty,a]$) instead of $\bR$. In that case we obtain the Banach Space $\til\cC_\varphi([a,\infty))$ (or $\til\cC_\varphi((-\infty,a])$). It is easy to construct an inclusion of $\til\cC_\varphi([a,\infty))$ into $\til\cC_\varphi$ using cutoff functions, so $\til\cC_\varphi([a,\infty))$ is a Banach subspace of $\til\cC_\varphi$.

It is important to point out that the function $\varphi$ given to define $\widetilde\cC^n_\varphi$ is not unique, in fact, we can always choose another one with better properties than the given $\varphi$.
\begin{thm}\label{phithm}\ 
\begin{enumerate}
\item For every $\varphi\in\cC^n(\bR,\bR^+)$ there exists $\psi\in\cC^\infty(\bR,\bR^+)$ such that  $\widetilde\cC^n_\varphi= \widetilde\cC^n_{\psi}$.
% Furthermore, if $\lils{t\to\pm\infty} 1/\varphi^{(j)}=0$, $j=1,\dots,n$, we can choose $\psi$ such that $\varphi/\psi\in\til\cC^n_1$.
\item Let $\varphi_1,\varphi_2\in\cC^n(\bR,\bR^+)$. If $\widetilde\cC^k_{\varphi_1}=\widetilde\cC^k_{\varphi_2}$ for some $k\in\{0,\dots,n\}$, then $\widetilde\cC^j_{\varphi_1}=\widetilde\cC^j_{\varphi_2}$ for every $j\in\{0,\dots,n\}$.
% Furthermore, if $\varphi_2/\varphi_1\in\til\cC^n_1$, then $\widetilde\cC^j_{\varphi_1}=\widetilde\cC^j_{\varphi_2}$ for every $j\in\{0,\dots,n\}$.
\end{enumerate}
\end{thm}
\begin{proof}
$1.$ For every $k\in\bZ$, let
\[\e_k:=\frac{\min \varphi|_{[k,k+1]}}{|k|+1}.\]
The Weierstrass Approximation Theorem guarantees the existence of $\varphi_k\in\cC^\infty([k,k+1],\bR^+)$ such that \[\|\varphi|_{[k,k+1]}-\varphi_k\|_\infty<\min\{\e_k,\e_{k-1}\}.\]
%
% Define then $\hat\varphi(t):=\varphi_k(t)$ for $t\in[k,k+1)$, $k\in\bZ$. $\varphi(t)$ is (generally speaking) discontinuous at $k\in\bZ$ and piecewise $\cC^\infty$.

Let $k\in\bZ$. We know that $\varphi$ is continuous at $k$, so there is $\d_k\in(0,1/2)$ such that $|\varphi(t)-\varphi(k)|<\min\{\e_k,\e_{k-1}\}$ for every $t\in[k-\d_k,k+\d_k]$.
Define 
\begin{align*} \rho(t):= & e^{-\frac{t^2}{1-t^2}}\(1-e^{-\frac{1}{t^2}}\),\ t\in(0,1);\ \rho(0):=1,\ \rho(1):=0. \end{align*}
It is easy to check that $\rho\in\cC^\infty([0,1],[0,1])$, $\rho^{(j)}(0)=\rho^{(j)}(1)=0$, $j\in\bN$.
Now consider the functions 
 \[\psi_k(t):=\begin{cases} \varphi_{k}(k+\d_k)+\int_{k+\d_k}^t\varphi'_{k}(s)\rho\(\frac{k+\d_k-s}{\c_{k,1}}\)\dif s, & t\in [a_k,k+\d_k],\\ \varphi_{k}(t), & t\in(k+\d_k,k+1-\d_{k+1}), \\ \varphi_{k}(k+1-\d_{k+1})+\int_{k+1-\d_{k+1}}^t\varphi'_{k-1}(s)\rho\(\frac{s-(k+1-\d_{k+1})}{\c_{k,2}}\)\dif s, & t\in [k+1-\d_{k+1},b_k],\end{cases}\]
for every $k\in\bN$, where $a_k:=k+\d_k-\c_{k,1}$, $b_k:=k+1-\d_{k+1}+\c_{k,2}$, $\c_{k,1}\in(0,\d_k)$ and  $\c_{k,2}\in(0,\d_{k+1})$. We have that $\psi_k\in\cC^\infty([a_k,b_k])$ and $\psi_k^{(j)}(a_k)=\psi_k^{(j)}(b_k)=0$  for every $j\in\bN$.
 
  Also, for $\c_{k,1}$ and $\c_{k,2}$ sufficiently small, we have that $|\varphi_{k}(t)-\psi_k(t)|<\max\{\e_{k},\e_{k-1}\}$ for  $t\in[a_k,b_k]$.
  
Hence, define
\[\psi(t):=\begin{dcases}
\psi_k(t), &  t\in [a_k,b_k),\ k\in\bZ,\\ 
\psi_k(b_k)+[\psi_{k+1}(a_{k+1})-\psi_k(b_k)]\rho\(\frac{t-a_{k+1}}{a_{k+1}-b_k}\), & t\in[b_k,a_{k+1}),\ k\in\bZ.\end{dcases}\]

$\psi\in\cC^\infty(\bR,\bR^+)$ and in the sets $[b_k,a_{k+1}]$, $k\in\bZ$, we have that
\begin{align*} & |\psi(t)-\varphi(t)|\\  = &  \left|\psi_k(b_k)+[\psi_{k+1}(a_{k+1})-\psi_k(b_k)]\rho\(\frac{t-a_{k+1}}{a_{k+1}-b_k}\)-\varphi(t)\right|\\ = & \left|\psi_k(b_k)-\varphi_k(b_k)+[\psi_{k+1}(a_{k+1})-\psi_k(b_k)]\rho\(\frac{t-a_{k+1}}{a_{k+1}-b_k}\)+\varphi_k(b_k)-\varphi(b_k)+\varphi(b_k)-\varphi(t)\right|\\ \le & |\psi_k(b_k)-\varphi_k(b_k)|+|\psi_{k+1}(a_{k+1})-\psi_k(b_k)|+|\varphi_k(b_k)-\varphi(b_k)|+|\varphi(b_k)-\varphi(t)|\\ < & 2\e_k+|\psi_{k+1}(a_{k+1})-\psi_k(b_k)|+|\varphi(b_k)-\varphi(k+1)|+|\varphi(k+1)-\varphi(t)|\\ \le & 4\e_k+|\psi_{k+1}(a_{k+1})-\varphi_{k+1}(a_{k+1})|+|\varphi_{k+1}(a_{k+1})-\varphi_{k+1}(b_{k})|+|\varphi_{k+1}(b_{k})-\psi_k(b_k)|\\
\le & 6\e_k+|\varphi_{k+1}(a_{k+1})-\varphi_{k+1}(b_{k})|
\\ \le & 6\e_k+|\varphi_{k+1}(a_{k+1})-\varphi(a_{k+1})|+|\varphi(a_{k+1})-\varphi(b_{k})|+|\varphi(b_{k})-\varphi_{k+1}(b_{k})|\\
  \le & 9 \e_k= 9\frac{\min \varphi|_{[k,k+1]}}{|k|+1}.\end{align*}
Therefore, \[|\psi(t)-\varphi(t)|<9 \frac{\min \varphi|_{[k,k+1]}}{|k|+1},\] for 
every $t\in[k,k+1]$, $k\in\bZ$.
Now,
\[|\psi(t)|>|\varphi(t)|-9 \frac{\min \varphi|_{[k,k+1]}}{|k|+1},\]
for every $t\in[k,k+1]$, $|k|>9$. Thus,
\begin{align*} \frac{|\varphi(t)-\psi(t)|}{|\psi(t)|}< & \frac{9 \min \varphi|_{[k,k+1]}}{(|k|+1)|\psi(t)|}<\frac{ 9\min \varphi|_{[k,k+1]}}{(|k|+1)\(|\varphi(t)|- \frac{9\min \varphi|_{[k,k+1]}}{|k|+1}\)}\\ = & \frac{ 9}{(|k|+1)\frac{|\varphi(t)|}{\min \varphi|_{[k,k+1]}}- 9}\le\frac{ 9}{|k|- 8}.\end{align*}
This fact allows us to prove that
\[\lil_{t\to \pm\infty}\left|\frac{\varphi(t)}{\psi(t)}-1\right|=\lil_{t\to \pm\infty}\frac{|\varphi(t)-\psi(t)|}{|\psi(t)|}\le\lil_{|k|\to \infty} \frac{ 9}{|k|- 8}=0.\]
Hence,
\[\lil_{t\to \pm\infty}\frac{\varphi(t)}{\psi(t)}=\lil_{t\to \pm\infty}\frac{\psi(t)}{\varphi(t)}=1.\]
Therefore, if $f\in\widetilde\cC^n_\varphi$,
\[\lim_{t\to \pm\infty}\frac{f(t)}{\psi(t)}=\lim_{t\to \pm\infty}\frac{f(t)}{\varphi(t)}\frac{\varphi(t)}{\psi(t)}=\lim_{t\to \pm\infty}\frac{f(t)}{\varphi(t)}\lim_{t\to \pm\infty}\frac{\varphi(t)}{\psi(t)}=\til f(\pm\infty).\]
Thus, $f\in\widetilde\cC^n_{\psi}$. The other inclusion is analogous, so $\widetilde\cC^n_\varphi= \widetilde\cC^n_{\psi}$.\par

2. By definition, $\varphi_2\in\til\cC^k_{\varphi_2}=\til\cC^k_{\varphi_1}$, so there exists $\til f_{\varphi_2}\in\cC^k(\ol\bR,\bR)$ such that $\varphi_2=\varphi_1\til f_{\varphi_2}|_{\bR}$. By Remark \ref{remext}, we have that $\til f_{\varphi_2}\in\cC^n(\ol\bR,\bR)$.
% Thus, $\til f_{\varphi_2}|_{\bR} =\varphi_2/\varphi_1\in\cC^n(\bR,\bR^+)$. Since $\til f_{\varphi_2}\in\cC^0(\ol \bR,\bR)$, we conclude, by the boundedness of the derivatives of $\til f_{\varphi_2}$ and the Mean Value Theorem, that $\til f_{\varphi_2}\in\cC^{n-1}(\ol \bR,\bR)$.

  Hence, for $j\in\{0,\dots,n\}$ and $f\in\til\cC^j_{\varphi_2}$, there exists  $\til f_{2}\in\cC^j(\ol\bR,\bR)$ such that $f=\varphi_2\til f_{2}|_{\bR}=\varphi_1\(\til f_{\varphi_2}\til f_{2}\)|_{\bR}$.  Therefore, we show that $f\in\til\cC^j_{\varphi_1}$ and thus $\til\cC^j_{\varphi_2}\subset \til\cC^j_{\varphi_1}$. Analogously, $\til\cC^j_{\varphi_1}\subset \til\cC^j_{\varphi_2}$ and so $\til\cC^j_{\varphi_2}=\til\cC^j_{\varphi_1}$.\par
%It is left the case where $j=n$.
%
%Using the same argument with the additional hypotheses, we deduce that $\til f_{\varphi_2}\in\cC^{n}(\ol \bR,\bR)$ and therefore $\til\cC^j_{\varphi_2}=\til\cC^j_{\varphi_1}$ for every $j\in\{0,\dots,k\}$.
\end{proof}
\begin{rem}  Theorem \ref{phithm} allows us to consider spaces of the form $\cC^n_\varphi$ even when $\varphi\in\cC(\bR,\bR^+)$ is not differentiable. In order to do so, we just pick a function $\psi\in\cC^\infty(\bR,\bR^+)$ to represent the space $\til\cC_\varphi=\til\cC_{\psi}$ and consider $\til\cC^n_{\psi}$. Furthermore, Theorem \ref{phithm} implies that $\til\cC^n_{\psi}$  does not depend  of the  choice of $\psi$.
\end{rem}

\section{Fixed points of integral equations}
Fix $\varphi\in\cC^n(\bR,\bR^+)$ and consider an operator $T$ given by\begin{equation}\label{eqhamm}
	Tu(t):=p(t)+\int_{-\infty}^{\infty} k(t,s)\eta(s)f(s,u(s))\dif s.
\end{equation}

We will obtain some results regarding to the existence of fixed points of operator $T$. To do that, we will follow the line of \cite{FigToj}, where the authors studied the existence of solutions of integral equations of Hammerstein-type in abstract cones. In particular, they considered a real normed space $(N,\|\cdot\|)$ and a continuous functional $\alpha\colon N\rightarrow \bR$. They proved that if this functional $\alpha$ satisfies the three following properties:
\begin{enumerate}
	\item[$(P_1)$] $\alpha(u+v)\ge \alpha(u)+ \alpha(v), \text{ for all } u,\,v \in N;$
	\item[$(P_2)$] $\alpha(\lambda\,u)\ge \lambda\,\alpha(u)$, for all $u\in N$, $\lambda\ge 0$;
	\item[$(P_3)$] $\left[\alpha(u)\ge 0, \ \alpha(-u)\ge 0\right] \Rightarrow u\equiv 0$; 
\end{enumerate}
then 
\[K_\alpha=\left\{u\in N\ : \ \alpha(u)\ge 0\right\}\]
is a cone.

Following their arguments, we will consider the cone 
\begin{displaymath}K_\alpha=\left\{u\in\widetilde{\cC}^n_\varphi\ : \ \alpha(u)\ge 0\right\},\end{displaymath}
where $\alpha\colon \widetilde{\cC}^n_\varphi \rightarrow \bR$ is a functional satisfying $(P_1)-(P_3)$.

Assume the following:

\begin{enumerate}
	\item [ $(C_{1})$] The kernel $k:\bR^2\to\bR$, is such that  $\frac{\partial ^j k}{\partial t^j}(t,\cdot)\,\eta(\cdot)\in\Lsp{1}(\bR)$   for every $t\in\bR$, $j=0,\dots,n$; $k(\cdot,s)\,\eta(s)\in\widetilde\cC_{\varphi}^n$ for every $s\in\bR$. Moreover, for every $\varepsilon >0$ and $j=0,\dots,n$, there exist $\delta>0$ and a measurable function $\omega_j$ such that if $|t_1-t_2|<\delta$ then \[\left|\frac{\partial ^j (k/\varphi)}{\partial t^j}(t_1,s)\,\eta(s)-\frac{\partial ^j (k/\varphi)}{\partial t^j}(t_2,s)\,\eta(s)\right| < \varepsilon\,\omega_j(s)\] for a.\,e. $s \in \mathbb{R}$.

\item  [ $(C_{2})$] $f\colon\bR^2\rightarrow [0,\infty)$  satisfies a sort of $\Lsp{1}$-Carath\'{e}odory 
conditions, that is, $f(\cdot,y)$ is measurable for each fixed
$y\in\bR$ and $f(t,\cdot)$ is continuous for a.\,a. $t\in \bR$ and, for each $r>0$, there exists $\phi_{r} \in \Lsp{1}(\bR)$ such that $f(t,y\,\varphi(t))\le \phi_{r}(t)$ for all $y\in [-r,r]$ and a.\,a. $t\in \bR$.

\item [$(C_{3})$] For every fixed $r>0$, $j=0,\dots,n$ and $l=0,\dots,j$,
\[\frac{\partial ^{j-l}}{\partial t^{j-l}}\frac{1}{\varphi}(t)\int_{-\infty}^{\infty} \left|\frac{\partial ^l k}{\partial t^l}(t,s)\,\eta(s)\right|\,\phi_r(s)\,\dif s\in \Lsp{\infty}(\bR)\]
and $\omega_j\,\phi_r \in \Lsp{1}(\bR)$.

Moreover, defining \[z^\pm(s):=\lim\limits_{t\rightarrow \pm \infty} \frac{k(t,s)\,\eta(s)}{\varphi(t)}\] and \[M(s):=\sup_{t\in\bR} \left|\frac{k(t,s)\,\eta(s)}{\varphi(t)}\right|,\] this functions must satisfy that $ \left|z^\pm\right|\,\phi_r, \, M\,\phi_r\in \Lsp{1}(\bR)$ for all $r>0$.

\item  [ $(C_{4})$] $p\in \widetilde{\cC}_\varphi^n$.
\item  [ $(C_{5})$] $\alpha(k(\cdot,s)\,\eta(s))\ge 0$ for a. e. $s\in\bR$ and $\alpha(p)\ge 0$.
\item  [ $(C_{6})$] \begin{displaymath}\alpha(Tu)\ge \int_{-\infty}^{\infty} \alpha(k(\cdot,s)\,\eta(s))\,f(s,u(s))\, \dif s +\alpha(p) \ \text{for all } u\in K_\alpha.\end{displaymath}
\item  [ $(C_{7})$] There exist two continuous functionals $\beta,\,\gamma\colon \widetilde\cC^n_\varphi \rightarrow \mathbb{R}$ satisfying that, for $u,\,v\in K_\alpha$ and $\lambda\in[0,\infty)$,
\[\beta(\lambda\,u)=\lambda\,\beta(u), \quad \beta(Tu)\le\int_{-\infty}^{\infty}\beta\left(k(\cdot,s)\,\eta(s)\right)\,f(s,u(s))\, \dif s +\beta(p).\]
and
\[\gamma(u+v)\ge\gamma(u)+\gamma(v), \quad \gamma(\lambda\,u)\ge\lambda\,\gamma(u), \quad \gamma(Tu)\ge\int_{-\infty}^{\infty}\gamma\left(k(\cdot,s)\,\eta(s)\right) \,f(s,u(s))\, \dif s + \gamma(p).\]

Moreover, for all $s\in\bR$, $\beta(k(\cdot,s)\,\eta(s)),\, \gamma(k(\cdot,s)\,\eta(s))\in \Lsp{1}(\bR)$ must be positive and such that
\[\int_{-\infty}^{\infty}\beta\left(k(\cdot,s)\,\eta(s)\right)\, \dif s+\frac{\beta(p)}{\rho},\, \, \int_{-\infty}^{\infty}\gamma\left(k(\cdot,s)\,\eta(s)\right)\, \dif s +\frac{\gamma(p)}{\rho}>0.\]
\item  [ $(C_{8})$] There exists $\xi\in K_\alpha\setminus\{0\}$ such that $\gamma(\xi)\ge 0$.
\item  [ $(C_{9})$] For every $\rho>0$ there exist either $b(\rho)>0$ such that $\beta(u)\le b(\rho)$ for every $u\in K_\alpha$ satisfying $\gamma(u)\le\rho$ or $c(\rho)>0$ such that $\gamma(u)\le c(\rho)$ for every $u\in K_\alpha$ satisfying $\beta(u)\le \rho$.
\end{enumerate}

\begin{thm}\label{thmk}
Assume hypotheses $(C_{1})$--$(C_{6})$. Then $T$ maps $(\widetilde\cC^n_\varphi,\|\cdot\|_\varphi)$ to itself, is continuous and compact and maps $K_\alpha$ to $K_\alpha$.
\end{thm}
\begin{proof}
\emph{$T$ maps $(\widetilde\cC^n_\varphi,\|\cdot\|_\varphi)$ to $(\widetilde\cC^n_\varphi,\|\cdot\|_\varphi)$:} 
Let $u\in \widetilde\cC^n_\varphi$. By $(C_1)$, we can use Leibniz's Integral Rule for generalised functions (see \cite[p. 484]{Jones}) to get
\[\frac{\partial ^j \widetilde{Tu}}{\partial t^j}(t)=\frac{\partial ^j (Tu/\varphi)}{\partial t^j}(t)=\int_{-\infty}^{\infty} \frac{\partial^j \, \left(k (\cdot,s)\,\eta(s)/\varphi\right)}{\partial t^j}(t)\,f(s,u(s))\dif s + \frac{\partial^j (p/\varphi)}{\partial t^j}(t).\]

On the other hand, from condition $(C_1)$, given $\varepsilon\in\bR^+$, there exists some $\delta\in\bR^+$ such that for $t_1,\,t_2 \in \mathbb{R}$, $|t_1-t_2|<\delta$ it is satisfied that
\begin{equation*}\begin{split}
		\left| \frac{\partial^j \,\widetilde{ k (\cdot,s)\,\eta(s)}}{\partial t^j}(t_1)- \frac{\partial^j\, \widetilde{ k(\cdot,s)\,\eta(s)}}{\partial t^j}(t_2) \right| = & \left| \frac{\partial^j (k(\cdot,s)\,\eta(s)/\varphi)}{\partial t^j}(t_1)- \frac{\partial^j (k(\cdot,s)\,\eta(s)/\varphi)}{\partial t^j}(t_2) \right| < \,\varepsilon \, \omega_j(s),
	\end{split}\end{equation*} 
	and since $p\in \til \cC^n_\varphi$,
	\begin{equation*}\begin{split}
			\left| \frac{\partial^j \,\til p}{\partial t^j}(t_1)- \frac{\partial^j\, \til p}{\partial t^j}(t_2) \right| = & \left| \frac{\partial^j (p/\varphi)}{\partial t^j}(t_1)- \frac{\partial^j (p/\varphi)}{\partial t^j}(t_2) \right| <\varepsilon.
		\end{split}\end{equation*} 
Therefore, from $(C_2)$,
\begin{equation}\label{cont_Tu}\begin{split}
\left|\frac{\partial ^j \widetilde{Tu}}{\partial t^j}(t_1)-\frac{\partial ^j \widetilde{Tu}}{\partial t^j}(t_2)\right|  \le & \int_{-\infty}^{\infty} \left| \frac{\partial^j (k(\cdot,s)\,\eta(s)/\varphi)}{\partial t^j}(t_1)- \frac{\partial^j (k(\cdot,s)\,\eta(s)/\varphi)}{\partial t^j}(t_2) \right|  f(s,u(s))  \dif s \\
& + \left| \frac{\partial^j (p/\varphi)}{\partial t^j}(t_1)- \frac{\partial^j (p/\varphi)}{\partial t^j}(t_2) \right| \\ \le & \, \varepsilon \left(\int_{-\infty}^{\infty} \omega_j(s) \, f(s, u(s)) \dif s +1 \right) \le \varepsilon \left(\int_{-\infty}^{\infty} \omega_j(s) \, \phi_{\|u\|_\varphi}(s) \dif s +1 \right),
\end{split}\end{equation}
and, since $\omega_j\,\phi_{\|u\|_\varphi}\in \Lsp{1}(\bR)$, the previous expression is upperly bounded by $\varepsilon\,c$ for some positive constant $c$. Hence, $\frac{\partial ^j \widetilde{Tu}}{\partial t^j}$ is continuous in $\bR$, that is, $\widetilde{Tu}\in\cC^n(\bR,\bR)$. It is left to see that there exists
\[\lim\limits_{t\to \pm\infty}\widetilde{Tu}(t)=\lim\limits_{t\to \pm\infty}\frac{Tu(t)}{\varphi(t)}=\lim\limits_{t\to \pm\infty}\frac{1}{\varphi(t)} \int_{-\infty}^{\infty} k(t,s)\,\eta(s)\,f(s,u(s))\, \dif s +\lim\limits_{t\to \pm\infty}\frac{p(t)}{\varphi(t)} \in \bR.\] 
Since $p, \, k(\cdot,s)\,\eta(s)\in \widetilde\cC_\varphi^n$ for all $s\in\bR$, there exist
\[\lim\limits_{t\to \pm\infty}\frac{p(t)}{\varphi(t)}\in \bR, \quad \lim\limits_{t\to \pm\infty}\frac{k(t,s)\,\eta(s)}{\varphi(t)}=z^\pm(s) \in \bR. \]

On the other hand,
\[\left|\frac{k(t,s)\,\eta(s)}{\varphi(t)} \,f(s,u(s))\right|\le   M(s) \,f(s,u(s))\le M(s) \, \phi_{\|u\|_\varphi}(s)  \ \text{ for all } t\in\bR\]
and, from $(C_3)$, $M \, \phi_{\|u\|_\varphi} \in \Lsp{1}(\bR)$. Thus, from Lebesgue's Dominated Convergence Theorem,
\[\lim\limits_{t\to \pm\infty}\frac{1}{\varphi(t)} \int_{-\infty}^{\infty} k(t,s)\,\eta(s)\,f(s,u(s))\, \dif s= \int_{-\infty}^{\infty} \lim\limits_{t\to \pm\infty} \frac{k(t,s)\,\eta(s)}{\varphi(t)} \,f(s,u(s))\, \dif s= \int_{-\infty}^{\infty} z^\pm(s) \,f(s,u(s))\, \dif s \]
and since 
\[\left|\int_{-\infty}^{\infty} z^\pm(s) \,f(s,u(s))\, \dif s\right| \le \int_{-\infty}^{\infty} \left|z^\pm(s)\right| \,f(s,u(s))\, \dif s \le \int_{-\infty}^{\infty} \left|z^\pm(s)\right| \,\phi_{\|u\|_\varphi}(s)\, \dif s<\infty,\]
we deduce that there exists $\lim\limits_{t\to \pm\infty}\frac{Tu(t)}{\varphi(t)}$ and consequently $Tu\in \widetilde\cC^n_\varphi$.

It is left to see that $Tu$ is bounded in $\|\cdot\|_\varphi$. Using the General Leibniz's Rule (for differentiation), it is clear that
\[\frac{\partial ^j \widetilde{Tu}}{\partial t^j}=\frac{\partial ^j (Tu/\varphi)}{\partial t^j}=\sum_{l=0}^j {j \choose l}\frac{\partial ^l Tu}{\partial t^l}\frac{\partial ^{j-l}}{\partial t^{j-l}}\frac{1}{\varphi}.\]
Moreover, from Leibniz's Integral Rule for generalised functions (\cite[p. 484]{Jones}),
\[\frac{\partial ^l Tu}{\partial t^l}(t)=\int_{-\infty}^{\infty} \frac{\partial ^l k}{\partial t^l}(t,s)\,\eta(s)\,f(s,u(s))\dif s + \frac{\partial^l p}{\partial t^l}(t).\]
Thus,
\begin{equation}\label{norm-Tu-til}\begin{split}
\left\|\frac{\partial ^j \widetilde{Tu}}{\partial t^j}\right\|_\infty=&\left\|\sum_{l=0}^j {j \choose l}\frac{\partial ^l Tu}{\partial t^l}\frac{\partial ^{j-l}}{\partial t^{j-l}}\frac{1}{\varphi}\right\|_\infty\le\sum_{l=0}^j {j \choose l}\left\|\frac{\partial ^l Tu}{\partial t^l}\frac{\partial ^{j-l}}{\partial t^{j-l}}\frac{1}{\varphi}\right\|_\infty\\
 = & \sum_{l=0}^j {j \choose l} \left\|\frac{\partial ^{j-l}}{\partial t^{j-l}}\frac{1}{\varphi}(t)\left(\int_{-\infty}^{\infty} \frac{\partial ^l k}{\partial t^l}(t,s)\eta(s)f(s,u(s))\dif s + \frac{\partial ^l p}{\partial t^l}(t)\right) \right\|_\infty \\ 
\le & \sum_{l=0}^j {j \choose l} \left( \left\|\frac{\partial ^{j-l}}{\partial t^{j-l}}\frac{1}{\varphi}(t)\int_{-\infty}^{\infty} \frac{\partial ^l k}{\partial t^l}(t,s)\eta(s)f(s,u(s))\dif s\right\|_\infty + \left\|\frac{\partial ^{j-l}}{\partial t^{j-l}}\frac{1}{\varphi}(t)\frac{\partial ^l p}{\partial t^l}(t) \right\|_\infty\right).
\end{split}\end{equation}
It is satisfied that
\begin{equation}\label{norm-Tu-til2}\begin{split}
\left|\frac{\partial ^{j-l}}{\partial t^{j-l}}\frac{1}{\varphi}(t)\int_{-\infty}^{\infty} \frac{\partial ^l k}{\partial t^l}(t,s)\eta(s)f(s,u(s))\dif s\right| \le & \, \frac{\partial ^{j-l}}{\partial t^{j-l}}\frac{1}{\varphi}(t)\int_{-\infty}^{\infty} \left|\frac{\partial ^l k}{\partial t^l}(t,s)\,\eta(s)\right|\,f(s,u(s))\dif s \\
\le & \,\frac{\partial ^{j-l}}{\partial t^{j-l}}\frac{1}{\varphi}(t)\int_{-\infty}^{\infty} \left|\frac{\partial ^l k}{\partial t^l}(t,s)\,\eta(s)\right|\, \phi_{\|u\|_\varphi}(s)\, \dif s,
\end{split}\end{equation}
and so, from \eqref{norm-Tu-til} and \eqref{norm-Tu-til2},
\begin{equation}\label{bound_Tu}\begin{split}
\left\|\frac{\partial ^j \widetilde{Tu}}{\partial t^j}\right\|_\infty \le & \, \sum_{l=0}^j {j \choose l} \left(\left\|\frac{\partial ^{j-l}}{\partial t^{j-l}}\frac{1}{\varphi}(t)\int_{-\infty}^{\infty} \left|\frac{\partial ^l k}{\partial t^l}(t,s)\,\eta(s)\right| \, \phi_{\|u\|_\varphi}(s)\,\dif s \right\|_\infty + \left\|\frac{\partial ^{j-l}}{\partial t^{j-l}}\frac{1}{\varphi}(t)\frac{\partial ^l p}{\partial t^l}(t) \right\|_\infty\right) \\ 
 <& \,\infty.
\end{split}\end{equation}

Therefore, $\|Tu\|_\varphi<\infty$. \par 
\emph{Continuity:} Let $\{u_n\}_{n \in \bN}$ be a sequence which converges to $u$ in $\widetilde{\mathcal{C}}_{\varphi}^n$. Then, there exists some $R\in\bR$ such that $\|u_n\|_\varphi\le R$ for all $n\in\bN$.

Moreover, $\lim\limits_{n\rightarrow\infty}\|u_n-u\|_\varphi=0$ implies that $\lim\limits_{n\rightarrow\infty}\|\frac{u_n}{\varphi}-\frac{u}{\varphi}\|_\infty=0$, from where we deduce that $\frac{u_n(s)}{\varphi(s)} \to \frac{u(s)}{\varphi(s)}$ for a.\,e. $s \in \mathbb{R}$. Therefore,  $u_n(s) \to u(s)$ for a.\,e. $s \in \mathbb{R}$ and we have, by virtue of  $(C_2)$, that $f(s,u_n(s)) \to f(s,u(s))$ for a.\,e. $s \in \mathbb{R}$. 

Reasoning analogously to the previous case, it is clear that
\begin{displaymath}
\left|\frac{\partial^j  \widetilde{Tu_n}}{\partial t^j}(t)\right| \le \sum_{l=0}^j {j \choose l}\left( \left\|\frac{\partial ^{j-l}}{\partial t^{j-l}}\frac{1}{\varphi}(t)\int_{-\infty}^{\infty} \left|\frac{\partial ^l k}{\partial t^l}(t,s)\eta(s)\right|\phi_{R}(s)\,\dif s\right\|_\infty +\left\|\frac{\partial ^{j-l}}{\partial t^{j-l}}\frac{1}{\varphi}\frac{\partial ^l p}{\partial t^l} \right\|_\infty\right)\end{displaymath}
for all $t\in\bR$ and we obtain, by application of Lebesgue's Dominated Convergence Theorem, that $Tu_n \to Tu$ in $\widetilde{\mathcal{C}}_{\varphi}^n$. Hence, operator $T$ is continuous.
	
\emph{Compactness:} Let $B \subset \widetilde{\mathcal{C}}_{\varphi}^n$ a bounded set, that is, $\|u\|_{\varphi} \le R$ for all $u \in B$ and some $R>0$. 
Then, in the upper bound of $\left\|\frac{\partial ^j \widetilde{Tu}}{\partial t^j}\right\|_\infty$ found in \eqref{bound_Tu} we can substitute $\phi_{\|u\|_\varphi}(s)$ by $\phi_{R}(s)$ and so we have found an upper bound which does not depend on $u$. Therefore it is clear that the set $T(B)$ is totally bounded.

On the other hand, taking into account the upper bound obtained in \eqref{cont_Tu}, we have that if $t_1,\, t_2\in\bR$ are such that $|t_1-t_2|<\delta$ then
\begin{equation*}
\left|\frac{\partial^j \widetilde{Tu}}{\partial t^j}(t_1)-\frac{\partial^j \widetilde{Tu}}{\partial t^j}(t_2)\right| \le  \varepsilon \left(\int_{-\infty}^{\infty} f(r, u(r)) \dif r +1 \right) \le \varepsilon \left(\int_{-\infty}^{\infty} \phi_R(r) \dif r +1 \right), \ j=0,\dots,n,
\end{equation*} 
and, since $\phi_R\in \Lsp{1}(\bR)$, we can conclude that $T(B)$ is equicontinuous. 

In conclusion, we derive, by application of Ascoli--Arzela's Theorem, that $T(B)$ is relatively compact in $\widetilde{\mathcal{C}}_{\varphi}^n$ and therefore $T$ is a compact operator.

\emph{$T$ maps $K_\alpha$ to $K_\alpha$:} It is an immediate consequence of conditions $(C_5)$ and $(C_6)$.
\end{proof}

Now we will give some conditions under which we can assure that the index of some subsets of $K_\alpha$ is 1 or 0. We will consider the following sets:
\begin{displaymath}K_\alpha^{\beta,\,\rho}:=\left\{u\in K_\alpha\ :\   0\le \beta(u)<\rho \right\},\end{displaymath}
\begin{displaymath}K_\alpha^{\gamma,\,\rho}:=\left\{u\in K_\alpha\ :\   0\le \gamma(u)<\rho \right\}.\end{displaymath}

We define now two functions $b,\,c\colon \bR^+\rightarrow \bR^+$ in the conditions of $(C_9)$:
\begin{displaymath}b(\rho):=\sup\left\{\beta(u)\ :\  \ u\in K_\alpha, \,  \gamma(u)<\rho \right\},\end{displaymath}
\begin{displaymath}c(\rho):=\sup\left\{\gamma(u)\ :\  \ u\in K_\alpha,  \, \beta(u)<\rho \right\}.\end{displaymath}

With these definitions, $K_\alpha^{\beta,\,\rho}\subset K_\alpha^{\gamma,\,c(\rho)}$ and $K_\alpha^{\gamma,\,\rho}\subset K_\alpha^{\beta,\,b(\rho)}$.

To prove that the index of some of these subsets is $1$ or $0$, we will use the following well-known sufficient conditions.

Let $K$ be a cone. If $\Omega\subset\bR^n$ is open we denote by $\overline{\Omega}$ and $\partial \Omega$, respectively,
its closure and its boundary. Moreover, we will note $\Omega_K=\Omega \cap K$, which is an open subset of $K$ in the relative topology.

\begin{lem}\label{lemind}
	Let $\Omega$ be an open bounded set with $0\in \Omega_{K}$ and $\overline{\Omega_{K}}\ne K$. Assume that $F:\overline{\Omega_{K}}\to K$ is
	a continuous compact map such that $x\neq Fx$ for all $x\in \partial \Omega_{K}$. Then the fixed point index\index{fixed point index} $i_{K}(F, \Omega_{K})$ has the following properties.
	\begin{itemize}
		\item[(1)] If there exists $e\in K\backslash \{0\}$ such that $x\neq Fx+\lambda e$ for all $x\in \partial \Omega_K$ and all $\lambda
		\ge0$, then $i_{K}(F, \Omega_{K})=0$.
		\item[(2)] If  $\mu x \neq Fx$ for all $x\in \partial \Omega_K$ and for every $\mu \geq 1$, then $i_{K}(F, \Omega_{K})=1$.
		\item[(3)] If $i_K(F,\Omega_K)\ne0$, then $F$ has a fixed point in $\Omega_K$.
		\item[(4)] Let $\Omega^{1}$ be open in $X$ with $\overline{\Omega^{1}}\subset \Omega_K$. If $i_{K}(F, \Omega_{K})=1$ and
		$i_{K}(F, \Omega_{K}^{1})=0$, then $F$ has a fixed point in $\Omega_{K}\backslash \overline{\Omega_{K}^{1}}$. The same result holds if
		$i_{K}(F, \Omega_{K})=0$ and $i_{K}(F, \Omega_{K}^{1})=1$.
	\end{itemize}
\end{lem}

\begin{lem}
	Assume that
	
	$(I_\rho^1)$ there exists $\rho>0$ such that
	\begin{displaymath}f^\rho \int_{-\infty}^{\infty} \beta\left(k(\cdot,s)\,\eta(s)\right)\, \dif s +\frac{\beta(p)}{\rho}<1,\end{displaymath}
	where
	\begin{displaymath}f^\rho=\sup\left\{\frac{f(t,u(t))}{\rho}\ :\ t\in\bR,\ u\in K_\alpha, \ \beta(u)=\rho\right\}.\end{displaymath}
	
	Then  $i_{K_\alpha}(T, K_\alpha^{\beta,\,\rho})=1$.
\end{lem}

\begin{proof}
	We will prove that $Tu\neq \mu\,u$ for all $u\in \partial K_\alpha^{\beta,\,\rho}$ and for every $\mu\ge 1$.
	
	Suppose, on the contrary, that there exist some $u\in \partial K_\alpha^{\beta,\,\rho}$ and $\mu\ge 1$ such that
	\begin{displaymath}\mu\,u(t)=\int_{-\infty}^{\infty}k(t,s)\,\eta(s)\,f(s,u(s))\, \dif s + p(t).\end{displaymath}
	
	Then, taking $\beta$ on both sides and using $(C_7)$, we get
	\begin{equation*}\begin{split}
			\mu\,\rho=& \,\mu\,\beta(u)=\beta(Tu)\le \int_{-\infty}^{\infty}\beta(k(\cdot,s)\,\eta(s))\,f(s,u(s))\, \dif s +\beta(p) \\
			\le & \,\rho\, \left( f^\rho \int_{-\infty}^{\infty}\beta\left(k(\cdot,s)\,\eta(s)\right)\, \dif s + \frac{\beta(p)}{\rho}\right) <\rho,
		\end{split}\end{equation*}
		which is a contradiction. Therefore we conclude the veracity of the result.
	\end{proof}

	\begin{lem}
		Assume that
		
		$(I_\rho^0)$ there exists $\rho>0$ such that
		\begin{displaymath}f_\rho \int_{-\infty}^{\infty} \gamma(k(\cdot,s)\,\eta(s))\, \dif s + \frac{\gamma(p)}{\rho}>1,\end{displaymath}
		where
		\begin{displaymath}f_\rho=\inf\left\{\frac{f(t,u(t))}{\rho}\ :\  t\in\bR,\ u\in K_\alpha, \gamma(u)=\rho\right\}.\end{displaymath}
		
		Then  $i_{K_\alpha}(T,K_\alpha^{\gamma,\,\rho})=0$.
	\end{lem}
	
	\begin{proof}
		We will prove that there exists $e\in K_\alpha^{\gamma,\,\rho}\setminus\{0\}$ such that $u\neq Tu+\lambda\,e$ for all $u\in \partial K_\alpha^{\gamma,\,\rho}$ and all $\lambda>0$.
		
		Let us take $e$ as in $(C_8)$ and suppose that there exist some $u\in\partial K_\alpha^{\gamma,\,\rho}$ and $\lambda>0$ such that 
		\begin{displaymath}u(t)=\int_{-\infty}^{\infty}k(t,s)\,\eta(s)\,f(s,u(s))\, \dif s+ p(t) + \lambda\,e(t).\end{displaymath} 
		Now, taking $\gamma$ on both sides and using $(C_7)$ and $(C_8)$,
		\begin{equation*}\begin{split}
		\rho&=\gamma(u)=\gamma(Tu+\lambda\,e)\ge \gamma(Tu)+\lambda\,\gamma(e) \ge \gamma(Tu) \ge \int_{-\infty}^{\infty}\gamma\left(k(\cdot,s)\,\eta(s)\right)\,f(s,u(s))\, \dif s + \gamma(p) \\
		 &\ge \rho\, f_\rho \int_{-\infty}^{\infty}\gamma\left(k(\cdot,s)\,\eta(s)\right)\, \dif s +\gamma(p)>\rho,
		\end{split}\end{equation*}
		which is a contradiction.
	\end{proof}
	
	From the previous Lemmas, it is possible to formulate the following Theorem. In this case, we establish conditions to ensure the existence of one or two solutions of the integral equation \eqref{eqhamm}. However, similar results can be formulated to ensure the existence of three or more solutions.
	
	\begin{thm}
		The integral equation \eqref{eqhamm} has at least one nontrivial solution in $K_\alpha$ if one of the following conditions hold
		\begin{itemize}
			\item[$(S_1)$] There exist $\rho_1,\,\rho_2\in (0,\infty)$ with $\rho_2>b(\rho_1)$ such that $(I^0_{\rho_1})$ and $(I^1_{\rho_2})$ hold.
			\item[$(S_2)$] There exist $\rho_1,\,\rho_2\in (0,\infty)$ with $\rho_2>c(\rho_1)$ such that $(I^1_{\rho_1})$ and $(I^0_{\rho_2})$ hold.
		\end{itemize}
		The integral equation \eqref{eqhamm} has at least two nontrivial solutions in $K_\alpha$ if one of the following conditions hold
		\begin{itemize}
			\item[$(S_3)$] There exist $\rho_1,\,\rho_2,\,\rho_3\in (0,\infty)$ with $\rho_2>b(\rho_1)$ and $\rho_3>c(\rho_2)$ such that $(I^0_{\rho_1})$, $(I^1_{\rho_2})$ and $(I^0_{\rho_3})$ hold.
			\item[$(S_4)$] There exist $\rho_1,\,\rho_2,\,\rho_3\in (0,\infty)$ with $\rho_2>c(\rho_1)$ and $\rho_3>b(\rho_2)$ such that $(I^1_{\rho_1})$, $(I^0_{\rho_2})$ and $(I^1_{\rho_3})$ hold.
		\end{itemize}
	\end{thm}

	\begin{rem}\label{rem-int}
		We note that the previous results could also be formulated for $\til\cC_\varphi([a,\infty))$ or $\til\cC_\varphi((-\infty,a]))$ for any $a\in\bR$.
		
		Furthermore, we could also formulate previous theory in the space $\widetilde\cC_{\varphi,a,b}^n$. In particular, it is clear that, under hypotheses $(C_{1})-(C_{6})$, $T$ maps $\widetilde\cC_{\varphi,a,b}^n$ to itself, is continuous and compact. 
	\end{rem}

\section{An example}
We will finally apply the theory in the previous sections to a modification of problem \eqref{eqorr2}. As stated in the introduction, we ignore the friction term (the term depending on $u'$) because it is only related to atmospheric drag and therefore does not affect the asymptotic behavior

Hence, we study the problem
\begin{equation}\label{eqorr3}
u''(t)=f(t,u(t)),\ t\in[0,\infty);\quad u(0)=0,\ u'(0)=v_0,
\end{equation}
with $f:[0,\infty)^2\to\bR$ defined as
\[f(t,y)=-\frac{g\,R^2}{(y+R)^2}+h(t,y),\]
$h:[0,\infty)^2\to\bR$. Given the domain of $f$ and $h$ and taking into account Remark \ref{rem-int}, we will work on the interval $[0,\infty)$.

Rewriting \eqref{eqorr3} as an integral problem, we know that the solutions of \eqref{eqorr3} coincide with the fixed points of the following operator,
\begin{equation}
Tu(t)=p(t)+\int_{0}^{\infty}k(t,s)\,f(s,u(s))\,\dif s,
\end{equation}
where
\[p(t)=v_0\,t\]
and
\[k(t,s)=\left\{\begin{array}{ll}
t-s, & 0\le s\le t, \\[.2cm]
0, & \text{otherwise},
\end{array} \right.\]
is the corresponding Green's function. We note that in this case $k(t,s)\ge 0$ on $[0,\infty)^2$.

We  take 
\[h(s,y)=\frac{g\,R^2}{(y+R)^2}+y\,e^{-s}\]
for $s,y\in [0,\infty)$.

To ensure the constant sign of $f$, we extend $h$ (and thus $f$) in the following way:
\begin{equation*}
h(s,y)=\frac{g\,R^2}{(y+R)^2} \text{ for } y<0.
\end{equation*}

We will consider
\[\varphi(t)=t+1,\]
and  work in the space $\til\cC_\varphi([0,\infty))$. Our cone \[K_\alpha=\left\{u\in \til\cC_\varphi([0,\infty)) \ \colon \ \alpha(u)\ge 0\right\}\] will be defined by the functional
\begin{displaymath}\alpha(u)=\int_{0}^{\infty}\frac{u(t)}{\varphi_2(t)}\,\dif t-\|u\|_{\varphi_3},\end{displaymath}
with $\varphi_2(t)=c\,e^t$ for some constant $c>0$, which will be calculated later, and $\varphi_3(t)=e^t$.

The functional $\alpha$ is well defined because if $u\in K_\alpha$, then $u\in \widetilde{\cC}_\varphi$, that is, $u(t)=(t+1) \, \tilde{u}(t)$, with $\tilde{u} \in \cC(\ol \bR,\bR)$, which implies that $\tilde{u}$ is uniformly bounded for some constant $N$. Then,
\begin{equation*}\begin{split}
\left|\int_{0}^{\infty} \frac{u(t)}{c\,e^t} \dif t\right| &= \left|\int_{0}^{\infty} \frac{(t+1)\, \tilde{u}(t)}{c\,e^t} \dif t\right| \le \int_{0}^{\infty} \frac{t \, |\tilde{u}(t)|}{c\,e^t} \dif t +\int_{0}^{\infty} \frac{|\tilde{u}(t)|}{c\,e^t} \dif t \\
&\le N \left(\int_{0}^{\infty} \frac{t}{c\,e^t} \dif t +\int_{0}^{\infty} \frac{1}{c\,e^t} \dif t \right) = \frac{2\,N}{c},\end{split}\end{equation*}
and
\begin{displaymath}\sup_{t\in[0,\infty)}\frac{|u(t)|}{e^t} \le \sup_{t\in[0,\infty)}\frac{|u(t)|}{t+1}=\|u\|_\varphi,\end{displaymath}
so $\alpha(u)\in \bR$ for all $u\in K_\alpha$.

Moreover, it is immediate to check that $\alpha$ satisfies properties $(P_1)-(P_3)$ and therefore the cone $K_\alpha$ is well defined.

We will see now that hypothesis $(C_1)$-$(C_9)$ for $n=0$ are satisfied:

\begin{itemize}
	\item[$(C_1)$] In this case $\eta\equiv 1$ and $k(t,\cdot)\,\eta(\cdot)\in \Lsp{1}(\bR)$ for every $t\in\bR$; indeed
	\begin{equation*}
	\int_{0}^{\infty}\left|k(t,s)\,\eta(s)\right|\dif s = \int_{0}^{t} (t-s) \dif s = \frac{t^2}{2}.
	\end{equation*}
	Moreover, $k(\cdot,s)\,\eta(s)\in\til\cC_\varphi$ for every $s\in\bR$ since $k(\cdot,s)\,\eta(s)\in\cC(\bR)$ and there exist
	\[\lim\limits_{t\rightarrow \infty}\frac{k(t,s)\,\eta(s)}{\varphi(t)} = \lim\limits_{t\rightarrow \infty}\frac{t-s}{t+1}=1\]
	and
	\[\lim\limits_{t\rightarrow 0}\frac{k(t,s)\,\eta(s)}{\varphi(t)} = 0.\]
	Finally, we will see that last condition in $(C_1)$ is satisfied for $\omega_0(s)=1+s$.
	
	Fix $\varepsilon>0$. Since $\frac{1}{\varphi}$ is a uniformly continuous function, there exists $\delta<\varepsilon$ such that for $|t_1-t_2|<\delta$, $\left|\frac{1}{t_1+1}-\frac{1}{t_2+1} \right|<\varepsilon$. We will compute now the difference $\left|\frac{k(t_1,s)}{\varphi(t_1)}-\frac{k(t_2,s)}{\varphi(t_2)} \right|$. Fix $s\in [0,\infty)$,
	\begin{itemize}
		\item If $t_1,\,t_2>s$, then
		\begin{equation*}\begin{split}
		\left|\frac{k(t_1,s)}{\varphi(t_1)}-\frac{k(t_2,s)}{\varphi(t_2)} \right| & =\left|\frac{t_1-s}{t_1+1}-\frac{t_2-s}{t_2+1}\right| = \left|\frac{-1-s}{t_1+1}-\frac{-1-s}{t_2+1}\right|=(1+s)\left|\frac{1}{t_1+1}-\frac{1}{t_2+1} \right| \\  &<\varepsilon\,\omega_0(s).
		\end{split}\end{equation*}
		
		\item If $t_1>s$ and $t_2<s$, then
		\begin{equation*}\begin{split}
		\left|\frac{k(t_1,s)}{\varphi(t_1)}-\frac{k(t_2,s)}{\varphi(t_2)} \right| =\left|\frac{t_1-s}{t_1+1}\right|< \left|\frac{t_1-t_2}{t_1+1}\right|<\frac{\varepsilon}{t_1+1}<\varepsilon<\varepsilon\,\omega_0(s).
		\end{split}\end{equation*}
		
		\item If $t_1,\,t_2<s$, then
		\begin{equation*}\begin{split}
		\left|\frac{k(t_1,s)}{\varphi(t_1)}-\frac{k(t_2,s)}{\varphi(t_2)} \right|=0.
		\end{split}\end{equation*}
	\end{itemize}
		
	\item[$(C_2)$] By definition of $h$, we have that $f(t,y)=0$ for $y<0$ and $f(t,y)=y\,e^{-t}\ge 0$ for $y\ge 0$. Clearly, $f(\cdot,y)$ is measurable for each fixed $y\in\bR$ and $f(t,\cdot)$ is continuous for a.\,a. $t\in\bR$. Finally, for each $r>0$, $f(t,y\varphi(t))=0$ for all $y\in[-r,0]$ and
	\[f(t,y\varphi(t))=y\,\varphi(t)\,e^{-t}\le r\,\varphi(t)\,e^{-t} \]
	for all $y\in[0,r]$. Therefore condition $(C_2)$ is satisfied if we take $\phi_r(t)=r\,\varphi(t)\,e^{-t}$.
	
	\item[$(C_3)$] For a fixed $R\in \bR$, we have that
	\begin{equation*}
	\frac{1}{\varphi(t)}\int_{0}^{\infty}\left|k(t,s)\,\eta(s)\right|\phi_R(s)\dif s = \frac{1}{t+1}\int_{0}^{t} (t-s) \, R \,(s+1) \, e^{-s}\dif s = \frac{R}{t+1}\, (-3+2t+e^{-t}(3+t)),
	\end{equation*}
	so $\frac{1}{\varphi(t)}\int_{0}^{\infty}\left|k(t,s)\,\eta(s)\right|\phi_R(s)\dif s \in \Lsp{\infty}(\bR)$. Moreover,
	\[\int_{0}^{\infty}\omega_0(s)\,\phi_R(s)\dif s= \int_{0}^{\infty}R \,(s+1)^2 \, e^{-s}\dif s=5\,R, \]
	that is, $\omega_0\,\phi_R\in \Lsp{1}(\bR)$. 
	
	Finally, from the limits calculated in $(C_1)$ and the expression of Green's function, we have that $z^+(s)=1$, $z^-(s)=0$ and $M(s)=1$, so it is clear that $\left|z^+\right|\,\phi_R, \, \left|z^-\right|\,\phi_R, \,M\,\phi_R\in \Lsp{1}(\bR)$. 
	
	\item[$(C_4)$] It is clear that $p(t)=v_0\, t \in \widetilde \cC_\varphi$ since $p\in\cC(\bR)$ and there exist
	\[\lim\limits_{t\rightarrow \infty} \frac{p(t)}{\varphi(t)}=v_0 \]
	and
	\[\lim\limits_{t\rightarrow 0} \frac{p(t)}{\varphi(t)}=0. \]
	
	\item[$(C_5)$] We have to prove that 
	\begin{displaymath}\alpha(k(\cdot,s))=\int_{0}^{\infty}\frac{k(\tau,s)}{\varphi_2(\tau)} \dif \tau -\|k(\cdot,s)\|_{\varphi_3} \ge 0 \quad \text{for a. e. } s\in \bR.\end{displaymath}
	We have that
	\begin{equation*}\begin{split}
	\int_{0}^{\infty}\frac{k(\tau,s)}{\varphi_2(\tau)} \dif \tau &= \int_{s}^{\infty} \frac{\tau-s}{c\, e^\tau} \dif \tau  = \frac{e^{-s}}{c}.
	\end{split}\end{equation*}
	
	On the other hand, fixed $s$, we have that 
	\[\left|\frac{k(t,s)}{e^t}\right|=0, \quad t\le s, \]
	and
	\[\left|\frac{k(t,s)}{e^t}\right| = \frac{t-s}{e^t} = e^{-s}\, \frac{t-s}{e^{t-s}} \le e^{-s}\,e^{-1}, \quad t\ge s. \]
	
	Therefore, it is enough to take $c\le e$ to ensure that $\alpha(k(\cdot,s)) \ge 0$.
	
	On the other hand, 
	\[\alpha(p)=\int_{0}^{\infty} \frac{p(t)}{\varphi_2(t)} \dif t -\|p\|_{\varphi_3}= \int_{0}^{\infty} \frac{v_0\,t}{c\,e^t} \dif t-\max_{t\in[0,\infty)}\left|\frac{v_0\,t}{e^t}\right|= \frac{v_0}{c}-v_0\,e^{-1}.\]
	
	Therefore, $\alpha(p)\ge 0$ if and only if $c\le e$.
	
	\item[$(C_6)$] By definition,
	\begin{equation*}\begin{split}
	\alpha(Tu)=&\int_{0}^{\infty}\frac{Tu(t)}{\varphi_2(t)} \dif t-\|Tu\|_{\varphi_3}.
	\end{split}\end{equation*}
	We have that
	\begin{equation*}\begin{split}
	\int_{0}^{\infty}\frac{Tu(t)}{\varphi_2(t)} \dif t = & \int_{0}^{\infty}\left(\int_{0}^{\infty}\frac{k(t,s)}{\varphi_2(t)} f(s,u(s)) \dif s + \frac{p(t)}{\varphi_2(t)}\right) \dif t \\
	= & \int_{0}^{\infty} \left(\int_{0}^{\infty}\frac{k(t,s)}{\varphi_2(t)} \dif t\right) f(s,u(s)) \dif s +\int_{0}^{\infty} \frac{p(t)}{\varphi_2(t)} \dif t,
	\end{split}\end{equation*}
	and
	\begin{equation*}\begin{split}
	\|Tu\|_{\varphi_3}= &\left\|\int_{0}^{\infty}k(\cdot,s) \,f(s,u(s)) \dif s + p \right\|_{\varphi_3} \le \left\|\int_{0}^{\infty}k(\cdot,s) \, f(s,u(s)) \, \dif s \right\|_{\varphi_3} + \|p\|_{\varphi_3} \\
	\le & \int_{0}^{\infty} \|k(\cdot,s)\|_{\varphi_3} \,f(s,u(s)) \,\dif s + \|p\|_{\varphi_3},
	\end{split}\end{equation*}
	and, consequently,
	\begin{equation*}\begin{split}
	\alpha(Tu)\ge & \int_{0}^{\infty} \left(\int_{0}^{\infty}\frac{k(t,s)}{\varphi_2(t)} \dif t\right) f(s,u(s)) \dif s -\int_{0}^{\infty} \|k(\cdot,s)\|_{\varphi_3} \, f(s,u(s)) \dif s +\int_{0}^{\infty} \frac{p(t)}{\varphi_2(t)} \dif t -\|p\|_{\varphi_3} \\
	= & \int_{0}^{\infty}\alpha(k(\cdot,s)) \,f(s,u(s)) \dif s +\alpha(p).
	\end{split}\end{equation*}

	\item[$(C_7)$] We will define $\beta, \gamma \colon \widetilde \cC_\varphi \rightarrow \bR$ in the following way:
	\begin{displaymath}\beta(u)=\|u\|_{\varphi_3}, \text{ with } \varphi_3(t)=e^t,\end{displaymath}
	and
	\begin{displaymath}\gamma(u)=\int_{0}^{\infty}\frac{u(t)}{e^t} \, \dif t.\end{displaymath}
	Analogously to $\alpha$, functionals $\beta$ and $\gamma$ are well defined.
	
	Now, we will show that $\beta$ and $\gamma$ satisfy all the properties in condition $(C_7)$. It is obvious that $\beta(\lambda u)=\lambda\, \beta(u)$ for all $\lambda\in[0,\infty)$ and $u\in K_\alpha$.
	
	Moreover,
	\begin{equation*}\begin{split}
	\beta(Tu)=&\left\| Tu\right\|_{\varphi_3} \le \int_{0}^{\infty} \|k(\cdot,s)\|_{\varphi_3} \, f(s,u(s)) +\|p\|_{\varphi_3} =\int_{0}^{\infty} \beta(k(\cdot,s)) \,f(s,u(s))+\beta(p).
	\end{split}\end{equation*}
	
	Finally, it is clear that $\beta(k(\cdot,s))>0$ and, since $\beta(k(\cdot,s))=\sup\limits_{t\in[s,\infty)}\dfrac{t-s}{e^t}\le e^{-(s+1)}$,
	\begin{displaymath}0<\int_{0}^{\infty}\beta(k(\cdot,s)) \dif s \le \int_{0}^{\infty} e^{-(s+1)} \dif s=e^{-1},\end{displaymath}
	that is, $\beta(k(\cdot,s))\in L^1[0,\infty)$.
	
	With regard to $\gamma$, it is immediate that $\gamma$ is linear.
	
	Also, 
	\begin{equation*}\begin{split}
	\gamma(Tu)=&\int_{0}^{\infty}\frac{Tu(t)}{e^t} \,\dif t=\int_{0}^{\infty}\int_{0}^{\infty} \frac{k(t,s)}{e^t} \, f(s,u(s)) \,\dif s \,\dif t +\int_{0}^{\infty}\frac{p(t)}{e^t} \,\dif t \\
	= &\int_{0}^{\infty}\left(\int_{0}^{\infty} \frac{k(t,s)}{e^t}\dif t\right) \, f(s,u(s)) \, \dif s +\gamma(p) =\int_{0}^{\infty}\gamma(k(\cdot,s)) \, f(s,u(s)) \, \dif s +\gamma(p).
	\end{split}\end{equation*}

	Finally, 
	\begin{displaymath}\gamma(k(\cdot,s))=\int_{s}^{\infty}\frac{t-s}{e^t} \dif t = e^{-s}>0 ,\quad  s\in[0,\infty),\end{displaymath}
	and
	\begin{equation*}\begin{split}
	\int_{0}^{\infty}\gamma(k(\cdot,s)) \dif s =&\int_{0}^{\infty}\int_{0}^{\infty}\frac{k(t,s)}{e^t} \dif t \dif s =\int_{0}^{\infty} \int_{0}^{t} \frac{t-s}{e^t} \dif s \dif t= \int_{0}^{\infty}\frac{t^2}{2\,e^t} \dif t = 1,
	\end{split}\end{equation*}
	that is, $\gamma(k(\cdot,s))\in L^1[0,\infty)$.
	
	\item[$(C_8)$] By condition $(C_5)$ we know that $p \in K_\alpha\setminus\{0\}$. Since, 
	\begin{displaymath}\gamma(p)=\int_{0}^{\infty}\frac{p(t)}{e^t}\,\dif t= \int_{0}^{\infty}\frac{v_0\,t}{e^t}\,\dif t =v_0 >0,\end{displaymath}
	it is enough to take $\xi=p$.

	\item[$(C_9)$] Every $u\in K_\alpha$ satisfies that $\beta(u)\le \frac{1}{c} \, \gamma(u)$, so it is enough to define $b(\rho)=\frac{\rho}{c}$. 
\end{itemize}

Now, we will see that there exist some values of $\rho$ for which $\left(I^0_{\rho}\right)$ and $\left(I^1_{\rho}\right)$ are satisfied:

Let's take $u\in K_\alpha$ such that $\beta(u)=\rho$. Then
\[|u(t)|\le\rho\, e^t, \quad  t\in[0,\infty), \]
and
\[f^{\rho}\le \sup_{t\in[0,\infty)} \frac{\rho\, e^t \, e^{-t}}{\rho}=1.\]
Consequently,
\[f^{\rho} \int_{0}^{\infty}\beta(k(\cdot,s)) \, \dif s + \frac{\beta(p)}{\rho} \le  e^{-1} + \frac{v_0 \, e^{-1}}{\rho},\]
and $\left(I^1_{\rho}\right)$ is satisfied for all $\rho>\frac{ e^{-1}}{1-e^{-1}}v_0 = 0.58197\ldots \cdot v_0$.

On the other hand, $f_{\rho}\ge 0$ and so
\[f_{\rho} \int_{0}^{\infty}\gamma(k(\cdot,s)) \, \dif s + \frac{\gamma(p)}{\rho} \ge \frac{\gamma(p)}{\rho} = \frac{v_0}{\rho}.\]
Therefore, $\left(I^0_{\rho}\right)$ is satisfied for all $\rho<v_0$.

Finally, we will see that there exist $\rho_1,\,\rho_2\in (0,\infty)$ with $\rho_2>b(\rho_1)$ such that $\left(I^0_{\rho_1}\right)$ and $\left(I^1_{\rho_2}\right)$ hold.

We have proved that condition $(C_5)$ is satisfied for all $c\le e$. If we take $c\ge(1-e^{-1})\,e$ and we choose $\rho_1,\,\rho_2$ satisfying that $\rho_1<v_0$ and $\rho_2>v_0\,e^{-1}/(1-e^{-1}),$ 
then it is clear that $\left(I^0_{\rho_1}\right)$ and $\left(I^1_{\rho_2}\right)$ hold and 
\[b(\rho_1)<\frac{v_0}{c}\le \frac{v_0\,e^{-1}}{1-e^{-1}}<\rho_2. \]

Therefore, we conclude that problem \eqref{eqorr3} has at least a nontrivial solution in $K_\alpha$.


\begin{thebibliography}{10}
	\providecommand{\url}[1]{{#1}}
	\providecommand{\urlprefix}{URL }
	\expandafter\ifx\csname urlstyle\endcsname\relax
	\providecommand{\doi}[1]{DOI~\discretionary{}{}{}#1}\else
	\providecommand{\doi}{DOI~\discretionary{}{}{}\begingroup
		\urlstyle{rm}\Url}\fi
	
	\bibitem{Corduneanu}
	Corduneanu, C.: Integral Equations and Stability of Feedback Systems.
	\newblock Academic Press, New York (1973)
	
	\bibitem{DjeGue}
	Djebali, S., Guedda, L.: \emph{A third order boundary value problem with
		nonlinear growth at resonance on the half-axis}.
	\newblock Mathematical Methods in the Applied Siences  (2016)
	
	\bibitem{FiMinCar}
	Fialho, J., Minh\'os, F., Carrasco, H.: \emph{Singular and classical second
		order $\phi$-Laplacian equations on the half-line with functional boundary
		conditions}.
	\newblock Electronic Journal of Qualitative Theory of Differential Equations
	(2017)
	
	\bibitem{FigToj}
	Figueroa, R., Tojo, F.A.F.: \emph{Fixed points of Hammerstein-type equations on
		general cones}.
	\newblock arXiv preprint arXiv:1611.02487  (2016)
	
	\bibitem{Hardy}
	Hardy, G.H.: Orders of infinity, the `Infinit\"arcalc\"ul' of Paul Du
	Bois-Reymond.
	\newblock Cambridge Univ Press (1910)
	
	\bibitem{Holmes}
	Holmes, M.H.: Introduction to perturbation methods, vol.~20.
	\newblock Springer Science \& Business Media (2012)
	
	\bibitem{Jones}
	Jones, D.S.: The Theory of Generalised Functions.
	\newblock Cambridge Univ Press (2009)
	
	\bibitem{kelley}
	Kelley, J.L.: General topology, vol.~27.
	\newblock Springer Science \& Business Media (1975)
	
	\bibitem{Kuang}
	Kuang, Y.: Delay differential equations: with applications in population
	dynamics, vol. 191.
	\newblock Academic Press (1993)
	
	\bibitem{MinCar}
	Minh\'os, F., Carrasco, H.: \emph{Existence of Homoclinic Solutions for
		Nonlinear Second-Order Problems}.
	\newblock Mediterr. J. Math. \textbf{13}(3849) (2016)
	
	\bibitem{MinCar2}
	Minh\'os, F., Carrasco, H.: \emph{Unbounded Solutions for Functional Problems
		on the Half-Line}.
	\newblock Abstract and Applied Analysis  (2016)
	
	\bibitem{MinCar-4ord}
	Minh\'os, F., Carrasco, H.: \emph{Homoclinic solutions for nonlinear general
		fourth-order differential equations}.
	\newblock Mathematical Methods in the Applied Sciences  (2017)
	
	\bibitem{przerad}
	Przeradzki, B.: \emph{The existence of bounded solutions for differential
		equations in Hilbert spaces}.
	\newblock Annales Polonici Mathematici \textbf{LVI.2}, 103--121 (1992)
	
	\bibitem{SinMa}
	Singh, R.K., Manhas, J.S.: Composition operators on function spaces, vol. 179.
	\newblock Elsevier (1993)
	
	\bibitem{Sipser}
	Sipser, M.: Introduction to the Theory of Computation, vol.~2.
	\newblock Thomson Course Technology Boston (2006)
	
\end{thebibliography}
\end{document}